\documentclass[10pt]{article}
\usepackage[english]{babel}
\usepackage[utf8]{inputenc}
\usepackage{amsfonts}
\usepackage{amsmath,amssymb, amsthm}
\usepackage{enumerate}
\usepackage{verbatim}
\usepackage{tikz}
\usepackage[backend=bibtex,style=numeric]{biblatex}
 \bibliography{biblio} 

\newcommand{\citeN}[1]{\textcite{#1} }

\newtheorem{Problem}{Problem}
\newtheorem{theorem}{Theorem}
\newtheorem{proposition}{Proposition}
\newtheorem{Corollary}{Corollary}

 \title{Avoiding two consecutive blocks of same size and same sum over $\mathbb{Z}^2$}
 \author{Micha\"el Rao and Matthieu Rosenfeld}

\def\alphabet{\Sigma}
\def\restriction#1#2{{#1}_{[#2]}}

\begin{document}
\maketitle
\begin{abstract}
A long standing question asks whether $\mathbb{Z}$ is uniformly $2$-repetitive 
 [Justin 1972, Pirillo and Varricchio, 1994], that is, 
 whether there is an infinite sequence over a finite subset of $\mathbb{Z}$ avoiding two consecutive blocks of same size and same sum or not.
Cassaigne \emph{et al.} [2014]
showed that $\mathbb{Z}$ is not uniformly $3$-repetitive.
We show that $\mathbb{Z}^{2}$ is not uniformly $2$-repetitive.
Moreover, this problem is related to a question from M\"akel\"a in combinatorics on words and we answer to a weak version of it.
\end{abstract}
\section{Introduction}
Let $k\ge 2$ be an integer and $(G,+)$ a group. 
An \emph{additive $k$-th power} is a non empty word $w_1\ldots w_k$ over $\alphabet \subseteq G$ such that all for every $i\in\{2,\ldots, k\}$,
$\vert w_i\vert = \vert w_1 \vert$ and $\sum w_i = \sum w_1$ (where $\sum v = \sum_{i=1}^{\vert v\vert} v[i]$).
Using the terminology of 
\citeN{introadditivepowers} we say that a group $(G,+)$ is \emph{$k$-uniformly repetitive} if every infinite word over
a finite subset of $G$ contains an additive $k$-th power as a factor.
It is a long standing question whether $\mathbb{Z}$ is uniformly $2$-repetitive or not~\cite{Justin1972,introadditivepowers}.
\citeN{avoid3consblock} showed that there is an infinite word over the finite alphabet $\{0,1,3,4\} \subseteq \mathbb{Z}$ without
additive 3rd powers, that is  $\mathbb{Z}$ is not uniformly $3$-repetitive.
In Section \ref{sec:results} we show that:

\newtheorem*{Th:addsquareavoidable}{Theorem \ref{addsquareavoidable}}
\begin{Th:addsquareavoidable}
 $\mathbb{Z}^2$ is not uniformly $2$-repetitive.
\end{Th:addsquareavoidable}

When $(G,+)$ is the abelian-free group generated by the elements of $\Sigma$ we talk about abelian repetitions.
The avoidability of abelian repetitions has been studied since a question from Erd\H{o}s~\cite{erdos1,erdos2}. 
An \emph{abelian square} is any non-empty word $uv$ where $u$ and $v$ are permutations of each other. 
Erd\H{o}s asked whether there is an infinite abelian-square-free word over an alphabet of size 4. 
\citeN{keranen1} answered positively to Erd\H{o}s's question in 1992 by giving a 85-uniform morphism, found with the assistance of a computer, whose fixed point is abelian-square-free. 

Erd\H{o}s also asked if it is possible to construct a word over 2 letters which contains only small squares. 
\citeN{Entringer1974159} 
gave a positive answer to this question. 
They also showed that every infinite word over 2 letters contains arbitrarily long abelian squares.
This naturally leads to the following question from M\"akel\"a (see \cite{keranen2}): 
\begin{Problem}
\label{makquestsquares}
Can you avoid abelian squares of the form $uv$ where $|u|\geq 2$ over three letters ?  - Computer experiments show that you can avoid these patterns at least in words of length 450.
\end{Problem}
We show that the answer is positive if we replace 2 by 6:

\newtheorem*{Th:makcorr}{Theorem \ref{makcorr}}
\begin{Th:makcorr}
There is an infinite word over 3 letters avoiding abelian square of period more than $5$. 
\end{Th:makcorr}

The proofs of Theorem~\ref{addsquareavoidable} and Theorem~\ref{makcorr} are close in the spirit 
(in fact both theorems implies independently that $\mathbb{Z}^3$ is not $2$-repetitive in the terminology of \citeN{introadditivepowers}).
Moreover the proofs are both based on explicit constructions using the following morphism:
$$h_6:\left\{
  \begin{array}{llll}
      a \rightarrow & ace \hspace{1cm}& 
      b \rightarrow & adf\\ 
      c \rightarrow & bdf& 
      d \rightarrow & bdc\\ 
      e \rightarrow & afe& 
      f \rightarrow & bce.\\ 
    \end{array}
  \right.
$$

First, we need to show the following:
\newtheorem*{Th:abelianthe}{Theorem \ref{abelianthe}}
\begin{Th:abelianthe}
$h_6^\omega(a)$ is abelian-square-free.
\end{Th:abelianthe}

We describe in Section \ref{sec:main} an algorithm to decide if a morphic word avoids abelian powers, and use it to show Theorem~\ref{abelianthe}. 
This algorithm generalizes the previously known ones \cite{carpi1,fixedpointpowerfree}, and can decide on a wider class of morphisms which includes $h_6$. 
In Section \ref{sec:app}, we explain how to extend the decidability to additive and long abelian powers.
Finally, in Section \ref{sec:results}, we give the results and the constructions.

\section {Preliminaries}
We use terminology and notations of~\citeN{Lothaire}.
An \emph{alphabet} $\alphabet$ is a finite set of \emph{letters}, and a \emph{word} is a (finite or infinite) sequence of letters.
The set of finite words is denoted by $\alphabet^*$ and the empty word by $\varepsilon$. 
One can also view $\alphabet^*$ equipped with the concatenation as the free monoid over $\alphabet$.

For any word $w$, we denote by $|w|$ the length of $w$ and for any letter $a\in\alphabet$, $|w|_a$ is the number of occurrences of $a$ in $w$.
The \emph{Parikh vector} of a word $w\in \alphabet^*$, denoted by $\Psi(w)$, is the vector indexed by $\alphabet$ such that for every $a\in \alphabet$, $\Psi(w)[a]= |w|_a$. 
Two words $u$ and $v$ are \emph{abelian equivalent}, denoted by $u\approx_a v$, if they are permutations of each other, or equivalently if $\Psi(u)=\Psi(v)$.
For any integer $k\ge 2$, an \emph{abelian $k$-th power} is a word $w$ that can be written $w= w_1w_2\ldots w_k$ with for all $i\in\{2,\ldots, k\}$,
$w_i \approx_a w_1$.
Its \emph{period} is $\vert w_i\vert$.
An \emph{abelian square} (resp., \emph{cube}) is an abelian 2nd power (resp., abelian 3rd power).
A word is \emph{abelian-$k$-th-power-free}, or \emph{avoids abelian $k$-th powers}, if none of its non-empty factor is an abelian $k$-th power.

Let $(G,+)$ be a group and $\Phi: (\alphabet^*, . ) \rightarrow (G,+)$ be a morphism.
Two words $u$ and $v$ are \emph{$\Phi$-equivalent}, denoted $u\approx_\Phi v$, if $\Phi(u)=\Phi(v)$.
For any $k\ge 2$, a \emph{$k$-th power modulo $\Phi$} is a word $w=w_1w_2\ldots w_k$ with for all $i\in\{2,\ldots, k\}$,
$w_i \approx_\Phi w_1$. If moreover $|w_1|=|w_2| = \ldots = |w_k|$ then it is a \emph{uniform $k$-th power modulo $\Phi$}. 
A \emph{square modulo $\Phi$} (resp., \emph{cube modulo $\Phi$}) is a 2nd power (resp., 3rd power) modulo $\Phi$. 
In this article, we only consider groups $(G,+)=(\mathbb{Z}^d,+)$ for some $d>0$.
We say that $(G,+)$ is \emph{k-repetitive} (resp., \emph{uniformly k-repetitive}) if for any alphabet $\alphabet$ and any morphism 
$\Phi: (\alphabet^*, . ) \rightarrow (G,+)$ every infinite word over $\alphabet$ contains a $k$-power modulo $\Phi$ (resp., a uniform $k$-power modulo $\Phi$).
Note that, for any integers $n$ and $k$, if $(\mathbb{Z}^{n+1},+)$ is $k$-repetitive then $(\mathbb{Z}^n,+)$ is uniformly $k$-repetitive.
Uniform $k$-th powers modulo $\Phi$ are sometimes called \emph{additive $k$-th powers}, without mention of the morphism $\Phi$, if the value of $\Phi(a)$ is clear in the context.
$\Phi$ can be seen as a linear map from the Parikh vector of a word to $\mathbb{Z}^d$, therefore we can associate to $\Phi$ the matrix $F_\Phi$ such that $\forall w\in \alphabet^*$, $\Phi(w) = F_\Phi \Psi(w)$.
Note that if $d= |\alphabet|$ and $F_\Phi$ is invertible then two words are abelian-equivalent if and only if they are $\Phi$-equivalent.
An application of Szemerédi's theorem shows that for $d=1$, for any finite alphabet $\alphabet$ and $k\in \mathbb{N}$,
it is not possible to avoid $k$-th power modulo $\Phi$ over $\alphabet$, that is, $(\mathbb{Z},+)$ is $k$-repetitive for any $k$.
On the other hand, whether $\mathbb{Z}$ is uniformly $2$-repetitive or not is a long standing open question \cite{Justin1972,introadditivepowers},
and \citeN{avoid3consblock} showed that $\mathbb{Z}$ is not uniformly $3$-repetitive.
We show on Theorem \ref{addsquareavoidable} that  $\mathbb{Z}^2$ is not uniformly $2$-repetitive.

Let $\operatorname{Suff}(w)$ (resp., $\operatorname{Pref}(w)$, $\operatorname{Fact}(w)$) be the set of suffixes (resp., prefixes, factors) of $w$. 
For any morphism $h$, let $\operatorname{Suff}(h) = \cup_{a\in\alphabet}\operatorname{Suff}(h(a))$, $\operatorname{Pref}(h) = \cup_{a\in\alphabet}\operatorname{Pref}(h(a)) $ and $\operatorname{Fact}(h) = \cup_{a\in\alphabet}\operatorname{Fact}(h(a))$.

Let $\operatorname{Fact^\infty}(h)=\cup_{i= 1}^\infty \operatorname{Fact}(h^i)$, and let $h^\infty$ be the infinite words in the closure of $\operatorname{Fact^\infty}(h)$. 
Note that for every $h$ with an eigenvalue of absolute value larger than $1$, 
$h^\infty$ is not empty since $\operatorname{Fact^\infty}(h)$ is infinite.
A word from $h^\infty$ which is a fixed point of $h$ is a \emph{pure morphic word}.
A \emph{morphic word} is the image of a pure morphic word by a morphism.

To a morphism $h$ on $\alphabet^*$, we associate a matrix $M_h$ on $\alphabet \times \alphabet$ such that $(M_h)_{a,b}=\vert h(b)\vert _a$.
The \emph{eigenvalues of $h$} are the eigenvalues of $M_h$.

Let $h: \alphabet^* \rightarrow \alphabet^*$ be a morphism, 
we say that $h$ is \emph{primitive} if there exists $k\in \mathbb{N}$ such that for all $a \in \alphabet$, $h^k(a)$ contains all the letters of $\alphabet$ (that is, $M_{h^k}$ is positive).
We have the following proposition. 
\begin{proposition}\label{prop:factinfprim}
Let $h$ be a primitive morphism on $\alphabet$ with $\vert \alphabet\vert>1$. Then for every fixed point $w$ of $h$, $\operatorname{Fact}(w)=\operatorname{Fact^\infty}(h)$.
\end{proposition}
\begin{proof}
Since $h$ is primitive, for every $a\in \alphabet$, $a\in \operatorname{Fact}(w)$, and for every $k$, $\operatorname{Fact}(h^k(a))\subseteq \operatorname{Fact}(h^k(w))=\operatorname{Fact}(w)$.
Thus $\operatorname{Fact^\infty}(h)\subseteq \operatorname{Fact}(w)$.
On the other hand, let $k$ be such that $(M_{h^k})_{a,b}>0$ for every $a,b\in \alphabet$. Then the function $f_a(l)=\vert h^{lk}(a) \vert$ is strictly increasing, with $f_a(l+1)\ge \vert \alphabet\vert f_a(l)$, and for every $v\in \operatorname{Fact}(w)$, there is a $l$ such that $v\in \operatorname{Fact}(h^{lk}(w[1]))$, and thus $v\in \operatorname{Fact^\infty}(h)$.
\end{proof}

\def\vect#1{\mathbf{#1}}
\def\nullvect{\overrightarrow{0}}
\def\nullvectset{\mathopen{}\left\{{\scriptstyle\overrightarrow{0}}\right\}\mathclose{}}

\paragraph{k-templates}
A \emph{k-template} is a $(2k)$-tuple $t =[a_1, \ldots, a_{k+1}, \vect{d}_1, \ldots, \vect{d}_{k-1}]$
where $a_i\in \alphabet\cup \{\varepsilon\}$ and $\vect{d}_i \in \mathbb{Z}^n$.
A word $w= a_1 w_1 a_2 w_2 \ldots w_k a_{k+1}$, where $w_i\in \alphabet^*$, 
is a \emph{realization} of (or \emph{realizes})
the template $t$ if for all $i\in\{1, \ldots,k-1 \}, \ \Psi(w_{i+1})-\Psi(w_i) = \vect{d}_i$.
A word is then an abelian $k$-th power if and only if it realizes the $k$-template $[\varepsilon, \ldots,\varepsilon, \nullvect, \ldots, \nullvect]$.
A template $t$ is \emph{realizable} by $h$ if there is a word in $\operatorname{Fact^\infty}(h)$ which realizes $t$.

Let $h$ be a morphism and let $t' =[a'_1, \ldots, a'_{k+1}, \vect{d'}_1, \ldots, \vect{d'}_{k-1}]$ and $t =[a_1, \ldots, a_{k+1}, \vect{d}_1, \ldots, \vect{d}_{k-1}]$ be two $k$-templates.
We say that \emph{$t'$ is a parent by $h$ of $t$} if there are 
$p_1, s_1, \ldots , p_{k+1},s_{k+1}\in \alphabet^*$ such that:
\begin{itemize}
 \item $\forall i\in\{1,\ldots, k+1\}$, $h(a'_i)= p_ia_is_i$,
 \item $\forall i\in\{1,\ldots, k-1\}$, $\vect{d}_i = M_h \vect{d'}_i+ \Psi(s_{i+1}p_{i+2}) - \Psi(s_ip_{i+1})$.
\end{itemize}
We denote by $\operatorname{Par}_h(t)$ the set of parents by $h$ of $t$. 
Note that, by definition, for any $t'\in \operatorname{Par}_h(t)$ if $t'$ is realizable by $h$ then $t$ is realizable by $h$.
We show in Proposition \ref{hasancestors} that if $t$ is realized by a long enough word from $\operatorname{Par}_h(t)$ then there is a realizable pattern 
$t'\in \operatorname{Par}_h(t)$.

A template $t$ is \emph{an ancestor} of a template $t'$ if there exists $n\ge 1$ and a sequence of templates $t'=t_1, t_2,\ldots, t_n=t$ such that for any $i$, 
$t_{i+1}$ is a parent of $t_{i}$.
A template $t'$ is a \emph{realizable ancestor by $h$} of a template $t$ if $t'$ is an ancestor by $h$ of $t$ 
and if there is a word in $\operatorname{Fact^\infty}(h)$ which realizes $t'$.
For a template $t$, we denote by $\operatorname{Anc}_h(t)$ (resp., $\operatorname{Ranc}_h(t)$) the set of all the ancestors (resp., realizable ancestors) by $h$ of $t$.
Note that a template $t$ is realized by a word from $\operatorname{Fact^\infty}(h)$ if and only if $\operatorname{Ranc}_h(t)$ is not empty.
We may omit ``by $h$'' if the morphism is clear in the context.

In the rest of this section we recall some classical notions from linear algebra.
\paragraph{Jordan decomposition}
A Jordan block $J_n(\lambda)$ is a $n\times n$ matrix with $\lambda\in \mathbb{C}$ on the diagonal, $1$ on top of the diagonal and $0$ elsewhere.
\[J_{n}(\lambda)=
\left(
 \begin{array}{cccc}
  \lambda & 1\\
    & \lambda & 1& \text{\huge0}\\
     \text{\huge0} &  & \ddots & 1\\
    & & &  \lambda
 \end{array}
\right)
\]
We recall the following well known proposition (see~\cite{WKTbook}).
\begin{proposition}[Jordan decomposition] \label{Jordan}
For any $n\times n$ matrix $M$ on $\mathbb{C}$, there is an invertible $n\times n$ matrix $P$ and a $n\times n$ matrix $J$ such that
$M= P J P^{-1}$, and the matrix $J$ is as follows:
\[
\left(
 \begin{array}{cccc}
  J_{n_1}(\lambda_1)\\
    & J_{n_2}(\lambda_2) & & 0\\
     0 &  &\ddots \\
    & & &  J_{n_p}(\lambda_p)
 \end{array}
\right)
\]
where the $J_i$ are Jordan blocks. $P J P^{-1}$ is a \emph{Jordan decomposition} of $M$.
\end{proposition}
The $\lambda_i$, $i\in\{1,\ldots, p\}$, are the (non necessarily distinct) eigenvalues of $M$, and the $n_i$ are their corresponding algebraic multiplicities.

Note that for every $k\ge 0$, $(J_n(\lambda))^k$ is the $n\times n$ matrix $M$ with $M_{i,j}= \binom{k}{j-i} \lambda^{k-j+i}$, with $\binom{a}{b} = 0$ if $a<b$ or $b<0$.
Thus, if $\vert \lambda \vert < 1$, $\sum_{k=0}^\infty (J_n(\lambda))^k$ is the matrix $N$ with $N_{i,j} = (1-\lambda)^{i-j-1}$ if $j\ge i$, and $0$ otherwise.
We can easily deduce from these observations the series of $k$-th powers of a matrix in Jordan normal form, and its sum.

We introduce some additional notations used in Propositions \ref{bound<1} and \ref{bound>1}.
Given a square matrix $M$ and $PJP^{-1}$ a Jordan decomposition of $M$,
let $b:\{1,\ldots, n\}\rightarrow \{1,\ldots, p\}$ be the function that associates to an index $i$ of $M$ the number corresponding to its Jordan block in the matrix $J$, 
thus $\forall i\in \{1,\ldots, n\}$, $ \lambda_{b(i)} = J_{i,i}$.
Let $B$ be the map that associate to an index $i$ the submatrix corresponding to the Jordan block containing this index,
$\forall i\in \{1,\ldots, n\}$, $B(i) =  J_{n_{b(i)}}(\lambda_{b(i)})$.
For any vector $\vect{x}$ and $1\leq i_s\leq i_e\leq n$ such that $i_s$ is the index of the first row of a Jordan block and $i_e$ is the index of the last row of the same block,
we denote by $\restriction{\vect{x}}{i_s,i_e}$ the sub-vector of $\vect{x}$ starting at index $i_s$ and ending at index $i_e$ and then  
$\restriction{(J \vect{x})}{i_s,i_e}= B(i) \restriction{\vect{x}}{i_s,i_e}$.
The columns $i_s$ to $i_e$ from $P$ generate the \emph{generalized eigenspace} associated to the eigenvalue $\lambda_i$.
Let $E_c(M)$ be the \emph{contracting eigenspace} of $M$, that is, the subspace generated by columns $i$ of $P$ such that $\vert \lambda_{b(i)} \vert < 1$. 
Similarly let $E_e(M)$ be the \emph{expanding eigenspace} of $M$, that is, the subspace generated by columns $i$ of $P$ such that $\vert \lambda_{b(i)} \vert > 1$. 
Note that $E_c(M)$ and $E_e(M)$ are independent from the Jordan decomposition we chose.

\paragraph{Smith decomposition}
The Smith decomposition is useful to solve systems of linear Diophantine equations.

\def\Ker#1{\ker(#1)}

\begin{proposition}[Smith decomposition]\label{Smith}
For any matrix $M\in \mathbb{Z}^{n\times m}$, there are $U\in\mathbb{Z}^{n\times n}$, $D\in \mathbb{Z}^{n\times m}$ and $V\in \mathbb{Z}^{m\times m}$ such that:
\begin{itemize}
 \item $D$ is diagonal (\emph{i.e.} $D_{i,j}=0$ if $i\ne j$),
 \item $U$ and $V$ are unimodular (\emph{i.e.}, their determinant is $1$ or $-1$),
 \item  $M = UDV$.
\end{itemize}
\end{proposition}
The fact that $U$ and $V$ are unimodular tells us that they are invertible over the integers. 
If one want to find integer solutions $\vect{x}$ of the equation $M\vect{x}=\vect{y}$, where $M$ is an integer matrix and $\vect{y}$ an integer vector, one can use the Smith decomposition $UDV$ of $M$. 
One can suppose {w.l.o.g.} than $n=m$, otherwise, one can fill with zeros.
Then $DV\vect{x}=U^{-1}\vect{y}$. 
Integer vectors in $\Ker{M}$ form a lattice $\Lambda$. The set of columns $i$ in $V^{-1}$ such that $D_{i,i}=0$ gives a basis of $\Lambda$.
Let $\vect{y'}=U^{-1}\vect{y}$, which is also an integer vector. 
Finding the solution $\vect{x'}$ of $D\vect{x'}=\vect{y'}$ is easy, since $D$ is diagonal. 
The set of solutions is non-empty if and only if for every $i$ such that $D_{i,i}=0$, $\vect{y'}_i=0$ and $D_{i,i}$ divides $\vect{y'}_i$ otherwise. 
One can take $\vect{x_0}=V^{-1}\vect{x'_0}$ as a particular solution to $M\vect{x_0}=\vect{y}$, with $(\vect{x'_0})_i = 0$ if $D_{i,i}=0$, and $(\vect{x'_0})_i=\vect{y'}_i/D_{i,i}$ otherwise.
The set of solutions is then $\vect{x_0}+\Lambda$.

\bigskip

We denote by $||\vect{x}||$ the Euclidean norm of a vector $\vect{x}$.
For any matrix $M$, let $||M||$ be its norm induced by the Euclidean norm, 
that is $||M|| = \sup \left \{ \frac{||M\vect{x}||}{||\vect{x}||} : \vect{x} \ne \nullvect\right\}$. 
We will use the following classical Proposition from linear algebra (see~\cite{WKTbook}).
\begin{proposition}\label{WKT}
Let $M$ be a matrix, and let $\mu_{min}$ (resp., $\mu_{max}$) be the minimum (resp., maximum) over the eigenvalues of $M^*M$ (which are all real and non-negative).
Then for any $\vect{x}$:
$$\mu_{min}||\vect{x}||^2 \le ||M\vect{x}||^2 \le \mu_{max}||\vect{x}||^2.$$
\end{proposition}

\section{The Abelian-power-free case}\label{sec:main}
In this section, we show the following theorem.
\begin{theorem}\label{decidingtemplates}
  For any primitive morphism $h$ with no eigenvalue of absolute value 1 and any template $t_0$ it is possible to decide 
  if $\operatorname{Fact^\infty}(h)$ realizes $t_0$.
\end{theorem}

With the Proposition \ref{prop:factinfprim}, we get the following corollary.

\begin{Corollary}\label{decidingavoidab}
  For any primitive morphism $h$ with no eigenvalue of absolute value 1 it is possible to decide 
  if the fixed points of $h$ are abelian-$k$-th-power-free.
\end{Corollary}

The main difference with the algorithm from \citeN{fixedpointpowerfree} is that we allow $h$ to have eigenvalues of absolute value less than $1$, and this is required for the algorithms presented in Section~\ref{sec:app}.

The remaining of this section is devoted to the proof of Theorem \ref{decidingtemplates}.
The idea is to compute a finite set $S$ such that $\operatorname{Ranc}_h(t_0) \subseteq S \subseteq \operatorname{Anc}_h(t_0)$, and to show that if $w\in \operatorname{Fact^\infty}(h)$ realizes $t_0$ then there is a small factor of $w$ which realizes a template in $S$. 
Thus $\operatorname{Fact^\infty}(h)$ realizes $t_0$ if and only if a small factor realizes a template in $S$.

In this section, we take a primitive morphism $h$ on the alphabet $\alphabet$, and let $n=\vert \alphabet\vert$. 
Since the case $\vert \alphabet \vert =1$ is trivial, we suppose $n\ge 2$.
Moreover, we take a $k$-template $t_0$, for a $k\in \mathbb{N}$.
Let $M=M_h$ be the matrix associated to $h$, \emph{i.e.} $\forall i, j$, $M_{i,j}= |h(j)|_i$.
We have the following equality:
$$\forall w\in \alphabet^*,\ \Psi(h(w))= M \Psi(w).$$

We suppose that $M$ has no eigenvalue of absolute value 1 and that it has at least one eigenvalue of absolute value greater than 1.
From Proposition \ref{Jordan} there is an invertible matrix $P$ and a Jordan matrix $J$ such that $M= P J P^{-1}$.  
Thus $ P^{-1} M=  J P^{-1}$, and for any vector $\vect{x}$, $ P^{-1} M \vect{x}=  J P^{-1} \vect{x} $. 
We define the map $r$, such that $r(\vect{x})=P^{-1} \vect{x} $ and its projections $\forall i,\ r_i(\vect{x}) = (P^{-1} \vect{x})_i$.
Using this notation we have for any $w$, $r(\Psi(h(w)))= r(M\Psi(w))= Jr(\Psi(w))$. 

\paragraph{Bounds on the $P$ basis}
We show that for any vector $\vect{x}$ appearing on a realizable ancestor of any template $t_0$ and any $i$, $|r_i(\vect{x})|$ is bounded,
handling separately generalized eigenvectors of eigenvalues of absolute value less and more than $1$. 
It implies that there are finitely many such integer vectors, since columns of $P$ form a basis of $\mathbb{C}^n$. 

\begin{proposition}\label{bound<1}
For any $i$ such that $|\lambda_{b(i)}| <1$, $\{|r_i(\Psi(w))| : w \in \operatorname{Fact^\infty}(h)\}$ is bounded.
\end{proposition}
\begin{proof}
Take $i$ such that $|\lambda_{b(i)}| <1$, and let $i_s$ (resp., $i_e$) be the index that starts (resp., ends) the Jordan block $b(i)$ (thus $i_s\le i\le i_e$). 
Let $w$ be a factor of $\operatorname{Fact^\infty}(h)$. 
Then there is a factor $w'\in \operatorname{Fact}(h)$, an integer $l$ and for every $j \in \{0,\ldots, l-1\}$, a pair of words $(s_j, p_j) \in (\operatorname{Suff}(h),\operatorname{Pref}(h))$ such that:
$$w=\left( \prod_{j=0}^{l-1} h^j(s_j) \right) h^{l}(w') \left(\prod_{j=l-1}^0 h^j(p_j)\right).$$
Thus
$$r(\Psi(w)) = \sum_{j=0}^{l-1} J^j r(\Psi(s_j)) + J^{l} r(\Psi(w')) + \sum_{j=0}^{l-1} J^j r(\Psi(p_j))$$
and
$$r(\Psi(w))_{[i_s, i_e]} = \sum_{j=0}^{l-1}B(i)^j r(\Psi(s_jp_j))_{[i_s, i_e]} + B(i)^{l} r(\Psi(w'))_{[i_s, i_e]}.$$ 

Since $\lim_{l\to \infty} \left(\sum_{j=0}^{l} B(i)^j\right)$ exists, $\vert r_i(\Psi(w)) \vert$ is bounded.

More precisely, a bound for $\vert r_i(\Psi(w)) \vert $ can be found by the following way.
Working on the free group on $\alphabet$ and using the fact that $h$ is primitive, for every $l'> l$ one can find $a\in \{x^{-1} : x\in \alphabet\}$ and extend the sequence $(s_j,p_j)_{j\in \{0,\ldots, l-1\}}$ to the sequence $(s_j,p_j)_{j\in \{0,\ldots, l'-1\}}$ such that: 
$$w=\left( \prod_{j=0}^{l'-1} h^j(s_j) \right) h^{l'}(a) \left(\prod_{j=l'-1}^0 h^j(p_j)\right).$$
Thus there is an infinite sequence $(s_j,p_j)_{j\in \mathbb{N}}$ of elements in $(\operatorname{Suff}(h),\operatorname{Pref}(h))$ such that
$$r(\Psi(w))_{[i_s, i_e]}  = \sum_{j=0}^\infty B(i)^j r(\Psi(s_j p_j))_{[i_s, i_e]} .$$
For any $i$ such that $|\lambda_{b(i)}| <1$, $r_i(\Psi(w))$ is bounded by $\vect{u}\cdot \vect{v}$, where:
\begin{itemize}
\item $\vect{v}$ is the vector such that $\vect{v}_j= (1-\vert\lambda_{b(i)}\vert)^{i-j-1}$ if $j\in \{i,\ldots,i_e\}$, and zero otherwise,
\item $\vect{u}$ is the vector such that $\vect{u}_j\!=\!\max\{ \vert r_j(\Psi(sp)) \vert \!:\! (s,p) \in (\operatorname{Suff}(h),\operatorname{Pref}(h)) \}$.
 \qed 
\end{itemize}\end{proof}

Let $r_{i}^* = {2} \times \operatorname{max} \{|r_i(\Psi(w))| : w \in \operatorname{Fact^\infty}(h)\}$.
Let $\mathcal{R}_B$ be the set of templates $t =[a_1,\ldots, a_{k+1}, \vect{d}_1, \ldots, \vect{d}_{k-1}]$ such that for every $i$ with $\vert \lambda_{b(i)} \vert <1$ and $j\in\{1,\ldots , k-1\}$, $|r_i(\vect{d}_j)|\leq r_i^*$.

\begin{Corollary}
Every template which is realized by $h$ is in $\mathcal{R}_B$.
\end{Corollary}

We need a tight upper bound on $r_{i}^*$ for the algorithm to be efficient.
The bound from the last proposition could be too loose, but we can reach better bounds by considering the fact that (since $h$ is primitive) for any $l>1$, $h^l$ has the same factors than $h$.
For example, for the abelian-square-free morphism $h_8$ (Section \ref{sec:results:main}) the bound for the eigenvalue $\sim(0.33292,0.67077)$ is $5.9633$, and become $1.4394$ for 
$(h_8)^{20}$, while the observed bound on the prefix of size approximately 1 million of a fixed point of $(h_8)^2$ is $1.4341$.

\bigskip

For any template ${t_0}$, we denote by $\vect{X}_{t_0}$ the set of all the vectors that appear on an ancestor of $t_0$.

\begin{proposition}\label{bound>1}
For every $i$ such that $|\lambda_{b(i)}| >1$, for every template ${t_0}$, $\{ |r_i(\vect{x}) | : \vect{x}\in \vect{X}_{t_0} \}$ is bounded.
\end{proposition}
\begin{proof}
The proof is close to the proof of Proposition \ref{bound<1}.
Let $\vect{x}$ be a vector of $\vect{X}_{t_0}$. 
Then there is a vector $\vect{x}_0$ of $t_0$, an integer $l$ and 
for every $j \in \{0,\ldots, l-1\}$, a 4-uplet of words $(s_j,s'_j, p_j, p'_j) \in (\operatorname{Suff}(h),\allowbreak\operatorname{Suff}(h),\allowbreak\operatorname{Pref}(h),\allowbreak\operatorname{Pref}(h))$ such that:
$$\vect{x}_0 = \sum_{j=0}^{l-1} M^j \Psi(s_jp_j) + M^l \vect{x} - \sum_{j=0}^{l-1}  M^j \Psi(s'_jp'_j).$$ 
Thus $$r(\vect{x}_0) = \sum_{j=0}^{l-1} J^j r(\Psi(s_jp_j)-\Psi(s'_jp'_j)) + J^l r(\vect{x}).$$ 

Let $i_s$ (resp., $i_e$) be the starting (resp., ending) index of the block $b(i)$. 
Thus $$B(i)^l\restriction{r(\vect{x})}{i_s,i_e} = \restriction{r(\vect{x}_0)}{i_s,i_e} + \sum_{j=0}^{l-1} B(i)^j\restriction{r(\Psi(s'_jp'_j)-\Psi(s_jp_j))}{i_s,i_e} .$$
Moreover we know that $B(i)$ is invertible so:
$$\restriction{r(\vect{x})}{i_s,i_e} = B(i)^{-l}\restriction{r(\vect{x}_0)}{i_s,i_e} + \sum_{j=0}^{l-1} B(i)^{j-l}(\restriction{ r(\Psi(s'_jp'_j)-\Psi(s_jp_j))}{i_s,i_e} .$$

The only eigenvalue of $B(i)^{-1}$ is $\lambda_{b(i)}^{-1}$ and has absolute value less than $1$, thus $\sum_{j=1}^{\infty} ||B(i)^{-j}||$ converges.
Hence $||\restriction{r(\vect{x})}{i_s,i_e}||$ can be bounded by a constant depending only on $h$, $P$, $J$ and $i$.
Thus there is a constant $r_{i,t_0}^*$ such that for all $\vect{x}\in \vect{X}_{t_0}$ $|r_i(\vect{x})|\leq r_{i,t_0}^*$.
\end{proof}

As we will see on the next paragraph, we do not need to compute a value for the bound $r_{i,t_0}^*$.
Since the columns of $P$ is a basis, Propositions \ref{bound<1} and \ref{bound>1} imply that the norm of any vector of a template from $\mathcal{R}_B\cap \operatorname{Anc}_h(t_0)$ is bounded,
and thus $\mathcal{R}_B\cap \operatorname{Anc}_h(t_0)$ is finite. 
Moreover we know that  $\operatorname{Ranc}_h(t_0)\subseteq \mathcal{R}_B\cap \operatorname{Anc}_h(t_0)$ so we have the following corollary.
\begin{Corollary}\label{finddiscrset}
For any template $t_0$ and any morphism $h$ whose matrix has no eigenvalue of absolute value 1, $\operatorname{Ranc}_h(t_0)$ is finite.
\end{Corollary}
\paragraph{Computation of the parents and ancestors}
Propositions \ref{bound<1} and \ref{bound>1} give us a naive algorithm to compute a set $S$ of templates such that 
$\operatorname{Ranc}_h(t_0)\subseteq S\subseteq \operatorname{Anc}_h(t_0)$.
We first compute a set of templates $T_{t_0}$ whose vectors' coordinates in basis $P$ are bounded by $r_i^*$ or $r_{i,t_0}^*$, 
then we compute the parent relation inside $T_{t_0}$ and we select the parents that are accessible from $t_0$.
But this method is not efficient at all, since for morphisms whose fixed points avoid abelian powers, the set of ancestors $\mathcal{R}_B\cap \operatorname{Anc}_h(t_0)$ is usually very small relatively to $T_{t_0}$.

It is better to use the following algorithm to compute a super-set of realizable ancestors of $t_0$.
We compute recursively a set of templates $A_{t_0}$ that we initialize at $\{t_0\}$, and each time that we add a new template $t$,
we compute the set of parents of $t$ which are in $\mathcal{R}_B$ and add them to $A_{t_0}$.
At any time we have $A_{t_0}\subseteq \mathcal{R}_B\cap \operatorname{Anc}_h(t_0)$ which is finite so this algorithm terminates.
Moreover if a parent of a template is realizable then this template also is realizable. 
It implies that, at the end, $\operatorname{Ranc}_h(t_0)\subseteq A_{t_0}$.

We need to be able to compute the set of realizable parents of a template.
Let $t=[a_1, \ldots, a_{k+1}, \vect{d}_1, \ldots, \vect{d}_{k-1}]$ be a template, and 
assume that $t' =[a'_1, \ldots, a'_{k+1}, \vect{d'}_1, \ldots, \vect{d'}_{k-1}]$ is a parent of $t$, and $t'$ is realizable by $h$.
Then there are $p_1, s_1, \ldots , p_{k+1},s_{k+1}\in \alphabet^*$ such that:
\begin{itemize}
 \item $\forall i\in\{1,\ldots, k+1\}$, $h(a'_i)= p_ia_is_i$, 
 \item $\forall i\in\{1,\ldots, k-1\}$, $\vect{d}_i = M \vect{d'}_i+ \Psi(s_{i+1}p_{i+2}) - \Psi(s_ip_{i+1})$.
\end{itemize}
There are finitely many ways of choosing the $a'_i$ in $t'$ and finitely many ways of choosing the $s_i$ and the $p_i$,
so we only need to be able to compute the possible values of the $\vect{d'}_i$ of a template with fixed $a'_1, \ldots, a'_{k+1}$ and $s_1, p_1, \ldots , s_{k+1},p_{k+1}$.
(Note that this is easy if $M$ is invertible.)

Suppose we want to compute $\vect{d'}_m$ for some $m$.
That is, we want to compute all the integer solutions $\vect{x}$ of $M\vect{x}=\vect{v}$, where $\vect{v}=\vect{d}_m - \Psi(s_{m+1}p_{m+2}) + \Psi(s_m p_{m+1})$. 
If this equation has no integer solution, then the template $t$ has no parents with this choice of $a_i$, $p_i$ and $s_i$.
Otherwise, we use the Smith decomposition of $M$ to find a solution $\vect{x_0}$ and a basis $(\beta_1, ..., \beta_{\kappa})$ (where $\kappa= \dim \Ker{M}$) of the lattice $\Lambda = \Ker{M} \cap \mathbb{Z}^n$.
We are only interested in parents realizable by $h$, so we want to compute the set $\vect{X}=\{\vect{x}\in \vect{x_0} + \Lambda :\ \forall i \text{ s.t. } \vert \lambda_{b(i)} \vert < 1,\ |r_i(\vect{x})| \le r_i^*\}$. Since $\Lambda$ is included in the generalized eigenspace of the eigenvalue $0$, we know by Proposition \ref{bound<1} that $\vect{X}$ is finite.
Let $\mathcal{B}$ be the matrix whose columns are the elements of the basis $(\beta_1, ..., \beta_{\kappa})$, and let $\vect{X}_{\mathcal{B}}=\{\vect{x} \in \mathbb{Z}^\kappa:\ \vect{x_0} + \mathcal{B} \vect{x}\in \vect{X} \} $.
$\Ker M$ is generated by $\mathcal{B}$ but also by the generalized eigenvectors corresponding to a null eigenvalue which are columns of $P$. 
So there is a matrix $Q$ made of rows of $P^{-1}$ such that $Q \mathcal{B}$ is invertible.
All the rows of $Q$ are rows of $P^{-1}$ thus from Proposition \ref{bound<1} there are $c_1,\ldots, c_{\kappa}\in \mathbb{R}$ such that for any $\vect{x}\in \vect{X}_{\mathcal{B}}$ and $ i\in\{1,\ldots , \kappa\}$,
 $|(Q(\mathcal{B} \vect{x}+ \vect{x_0}))_i| \le c_i$ thus $|(Q \mathcal{B} \vect{x})_i| \le c_i+|(Q\vect{x_0})_i|$. 
Then: $$||Q \mathcal{B} \vect{x}||^2\leq \sum_{i=1}^{\kappa} (c_i + |Q \vect{x_0}|)^2 = c.$$
From Proposition \ref{WKT} if $\mu_{min}$ is the smallest eigenvalue of $(Q \mathcal{B})^* (Q \mathcal{B})$ then 
$\mu_{min}||\vect{x}||^2\leq||Q \mathcal{B} \vect{x}||^2\leq c$.
Moreover $Q \mathcal{B}$ is invertible, thus $\mu_{min} \not=0$, and $\vect{X}_{\mathcal{B}}$ contains only integer points in the ball of radius 
$\sqrt{\frac{c}{\mu_{min}}}$. 
We can easily compute a finite super-set of $\vect{X}_{\mathcal{B}}$, and thus of $\vect{X}$, and then we can select the elements that are actually in $\vect{X}$.
The choice of $\vect{x_0}$ is significant for the sharpness of the bound $c$: it is preferable to take a $\vect{x_0}$ nearly orthogonal to $\Ker{M}$.

\paragraph{Comparing to the factors}
Let $t=[a_1,\ldots, a_{k+1}, \vect{d}_1,\ldots, \vect{d}_{k-1}]$ be a $k$-template. 
Let $\Delta(t)= \max_{i=1}^{k-1} ||\vect{d}_i||_1$ and $\delta = \max_{a\in \alphabet} |h(a)|$.

\begin{proposition}\label{hasancestors}
Let $t$ be a $k$-template and $w\in\operatorname{Fact^\infty}(h)$ a word which realizes $t$.
If $\vert w \vert >  k\left(\frac{(k-1)\Delta(t)}{2} +\delta+1\right)+1$ then for every $w'\in\operatorname{Fact^\infty}(h)$ such that $w\in \operatorname{Fact}(h(w'))$ there is a parent $t'$ of $t$ such that $w'$ realizes $t'$.
\end{proposition}
\begin{proof}
Let $t= [a_1, \ldots, a_{k+1}, \vect{d}_1, \ldots, \vect{d}_{k-1}]$ be a $k$-template and $w\in\operatorname{Fact^\infty}(h)$ a word which realizes $t$ such that
$\vert w \vert >  k\left(\frac{(k-1)\Delta(t)}{2} +\delta+1\right)+1$.
Then there are $w_1, \ldots , w_n\in \alphabet^*$ such that $w =  a_1 w_1 a_2 w_2 \ldots w_k a_{k+1}$
and $\forall i\in\{1, \ldots,k-1 \}, \ \Psi(w_{i+1})-\Psi(w_i) = \vect{d}_i$.
Thus for any $i,j\in \{1,\ldots, k\}$ such that $j<i$, $\Psi(w_i)= \Psi(w_j)+\sum_{m=j}^{i-1}\vect{d}_m$ and, by triangular inequality, we have:
$$  \left|w_i\right|-\left|w_j\right|  = ||\Psi(w_i)||_1-||\Psi(w_j)||_1\leq \sum_{m=j}^{i-1}||\vect{d}_m||_1\leq (i-j)\Delta (t). $$
Therefore for any $i,j\in \{1,\ldots, k\}$, $  \left|w_j\right| \leq |i-j|\Delta (t) +\left|w_i\right|$,
and for any $i$,  $|w|\le \sum_{m=1}^{k}(|i-m|\Delta (t) +|w_i|) + k+1$. 
Thus $|w|\le \frac{k(k-1)}{2}\Delta(t) +k|w_i|+ k+1$.
Then $k\left(\frac{(k-1)\Delta(t)}{2} +|w_i|+1\right)+1\geq |w|> k\left(\frac{(k-1)\Delta(t)}{2} +\delta+1\right)+1$,
and consequently $\forall i,\ |w_i|> \delta = \max_{a\in \alphabet} |h(a)|$.
We also know that $\forall i$, $w_i \in \operatorname{Fact^\infty}(h)$ so there are $w'_1,\ldots, w'_k\in \alphabet^*$, 
$a'_1,\ldots, a'_{k+1}\in \alphabet$, $p_1,\ldots,p_{k+1}\in \operatorname{Pref}(h)$ and $s_1,\ldots, s_{k+1}\in \operatorname{Suff}(h)$ such that: 
\begin{itemize}
 \item $\forall i$, $w_i = s_i h(w'_i)p_{i+1}$,
 \item $\forall i$, $h(a'_i)= p_ia_is_i$,
 \item  $w'= a'_1w'_1a'_2\ldots a'_{k}w'_ka'_{k+1}$.
\end{itemize}
$w'$ realizes the template $t'= [a'_1, \ldots, a'_{k+1}, \Psi(w'_2)-\Psi(w'_1), \ldots, \Psi(w'_k)-\Psi(w'_{k-1})]$ which is a parent of $t$.
\end{proof}

\begin{proposition}\label{factorvsdiscrset}
Let $S$ be such that $\operatorname{Ranc}_h(t_0)\subseteq S \subseteq \operatorname{Anc}_h(t_0)$, and let $s=\max_{t\in S}  k\left(\frac{(k-1)\Delta(t)}{2} +\delta+1\right)+1$. Then the following are equivalent:
\begin{enumerate}
 \item there is a factor of $\operatorname{Fact^\infty}(h)$ of size at most $s$ realizing a template $t$ of $S$,
 \item there is a factor of $\operatorname{Fact^\infty}(h)$ realizing $t_0$.
\end{enumerate}
\end{proposition}
\begin{proof}
Let $w$ be a factor of $\operatorname{Fact^\infty}(h)$ such that the template $t_0$ is realized by $w$ and $|w|>s$.
By the definition of $\operatorname{Fact^\infty}(h)$, there are $l\in \mathbb{N}$ and $a\in \alphabet$ such that $w\in\operatorname{Fact}(h^l(a))$. 
Thus there is a sequence $w_0, w_1,\ldots, w_{l-1},w_l$ such that $\forall i, w_{i+1} \in \operatorname{Fact}(h(w_i))$, $w_0=a$ and $w_l=w$.
Then there is $i\in\{0,\ldots,l-1\}$ such that $|w_i|\le s$ and $\forall j>i$, $|w_j|> s$.
Consequently, from Proposition \ref{hasancestors}, $w_i$ realizes a template in $S$.

On the other hand, if a factor in $w'\in\operatorname{Fact^\infty}(h)$ realizes a template in $S$, then by the definition of the ancestors, there is a factor $w \in \operatorname{Fact^\infty}(h)$ realizing $t_0$.
\end{proof}

Since $h$ is primitive one can easily compute the factors of size $n+1$ from the set of factors of size $n$, 
and the set of factors of size 1 is $\alphabet$. 

We summarize the proof of Theorem \ref{decidingtemplates}.
We know that one can compute a set $S$ such that $\operatorname{Ranc}_h(t_0)\subseteq S \subseteq \operatorname{Anc}_h(t_0)$.
Moreover from Proposition \ref{factorvsdiscrset} we know that there is a $s$ such that the two following are equivalent:
\begin{enumerate}
 \item there is a factor of $\operatorname{Fact^\infty}(h)$ of size at most $s$ realizing a template $t$ of $S$,
 \item there is a factor of $\operatorname{Fact^\infty}(h)$ realizing $t_0$.
\end{enumerate}
The condition $1$ can be checked by a computer by generating all the factors of size less than $s$ and comparing them to all the element of $S$.
Hence one can decide if there is a factor of $\operatorname{Fact^\infty}(h)$ that realizes $t_0$.

\bigskip
In Section \ref{sec:results:main}, we present two new morphisms whose fixed points are abelian-square-free.
%
It would be interesting for the sake of completeness to be able to decide the abelian-$k$-th-power freeness for any morphism.
We can get ride of the primitivity condition with a lot of technicalities, but it seems much harder to deal with eigenvalues of absolute value exactly $1$
\begin{Problem}
  Is is decidable for any morphism $h$ if the fixed points of $h$ are abelian-$k$-th-power-free.
\end{Problem}

\section{Applications}\label{sec:app}

If a morphism $h$ has $k$ eigenvalues of absolute value less than $1$ (counting their algebraic multiplicities), then Proposition~\ref{bound<1} tells us that the Parikh vectors of the factors of $\operatorname{Fact}^\infty(h)$ are close to the subspace $E_e(M_h)$ of dimension $n-k$.
This can be useful to avoid patterns in images of $\operatorname{Fact^\infty}(h)$.

If one tries to avoid a template $t$ in a morphic word $g(h^\infty)$, with $g: \alphabet \to \alphabet'$ and $\vert \alphabet' \vert < \vert \alphabet \vert$, then the set of parents of $t$ is generally infinite: the set of the vectors in the parents is close to the subspace $\Ker{M_g}$ of dimension $\vert \alphabet \vert - \vert \alphabet' \vert$ (if $M_g$ has full rank).
But if the intersection of $\Ker{M_g}$ with $E_e(M_h)$ is of dimension $0$ then we can generate a finite super-set of the realizable parents, and decide with the algorithm from Section \ref{sec:main}. 

We can use the same idea to avoid additive powers.
This is a generalization of the method used in~\cite{avoid3consblock} to show that we can avoid additive cubes in a word over $\{0,1,3,4\}$.

We present here two applications of this method: decide if a morphic word does not contain large abelian powers
and decide if a pure morphic word avoids additive powers.
Other possible applications, such as deciding if a morphic word avoids $k$-abelian powers, are not explained here, but the method can be easily generalized.

\subsection{Deciding if a morphic word contains large abelian power}
In this subsection we take a second morphism $g$ and we want to decide whether the morphic word $g(h^\infty(a))$ avoids large abelian $k$-th powers.

\begin{proposition}\label{prop:largeperiodset}
If $M_h$ has no eigenvalue of absolute value $1$ and $E_e(M_h) \cap \Ker{M_g} = \nullvectset$, then for any template $t'$ one can compute a finite set $S$ that contains any template realizable by $h$ and parent of $t'$ by $g$.
\end{proposition}
\begin{proof}
The proof is similar to the computation of parents in Section~\ref{sec:main}.
Let $M_h=PJP^{-1}$ be a Jordan decomposition of $M_h$.
Let $\kappa =\dim \Ker{M_g}$ and $\Lambda = \Ker{M_g} \cap \mathbb{Z}^\kappa$.
We use the Smith decomposition of $M_g$ to get the matrix $B$, whose columns form a basis of $\Lambda$.
Assume $t =[a_1, \ldots, a_{k+1}, \vect{d}_1, \ldots, \vect{d}_{k-1}]$ is realizable by $h$ and parent of $t'=[a'_1, \ldots, a'_{k+1},\allowbreak \vect{d'}_1, \ldots, \vect{d'}_{k-1}]$ by $g$.
Then there are $p_1, s_1, \ldots , p_{k+1},s_{k+1}\in \alphabet^*$ such that:
\begin{itemize}
 \item $\forall i$, $g(a_i)= p_ia'_is_i$ 
 \item $\forall i$, $\vect{d'}_i = M_g \vect{d}_i+ \Psi(s_{i+1}p_{i+2}) - \Psi(s_ip_{i+1})$.
\end{itemize}
There are finitely many choices for the $a_i$, $s_i$ and $p_i$.
For a fixed choice of $a_i$, $s_i$ and $p_i$, we want to know all the possible values for $\vect{d}_m$ for some $m$ with fixed $a_1, \ldots, a_{k+1}$ and $p_1, s_1, \ldots , p_{k+1},s_{k+1}$. 
Then $\vect{d}_m$ is an integer solution of $M_g\vect{x}=\vect{v}$, with $\vect{v}=\vect{d'}_m+\Psi(s_mp_{m+1})-\Psi(s_{m+1}p_{m+2})$.
We will see that we have only finitely many choices for $\vect{d}_m$.
As already explained, if such a solution exists, then $\vect{d}_m\in \vect{x_0} + \Lambda$, and $\vect{x_0}$ can be found with the Smith decomposition of $M_g$.

Let $Q$ be the rectangular submatrix of $P^{-1}$ such that the $i$th line of $P^{-1}$ is a line of $Q$ if and only if $|\lambda_{b(i)}|<1$.
For every $\vect{x}\in \mathbb{C}^\kappa\setminus \nullvectset$, $B\vect{x} \in \Ker{M_g}$. 
Then, by hypothesis, $B\vect{x}\not\in E_e(M_h)$ and $QB\vect{x}\ne \nullvect$ since the lines of $Q$ generate the subspace orthogonal to $E_e(M_h)$.
Thus we have $ \operatorname{rank}(Q B) =\kappa $ which implies that there is a submatrix
$Q'$ of $Q $ such that $Q' B$ is invertible.

For all $i\in\{1,\ldots, \kappa\}$, let $p_i$ to be the function such that for all vector $\vect{x}$, $p_i(\vect{x}) = (Q' \vect{x})_i$.
From Proposition \ref{bound<1}, for all $ i\in\{1,\ldots, \kappa\}$, there is $c_i\in\mathbb{R}$ such that for any two factors $u$ and $v$ 
of $\operatorname{Fact^\infty}(h)$, $|p_i(\Psi(u)-\Psi(v))|\leq c_i$. 

Let $\vect{X}=\{\vect{x} \in \vect{x_0} + \Lambda : \forall i\in \{1,\ldots, \kappa\},\ |p_i(\vect{x})| \le c_i\}$.
Since we are only interested in realizable solutions, $\vect{d}_m$ has to be in $\vect{X}$.
Let $\vect{X}_B=\{\vect{x} \in\mathbb{Z}^\kappa:\ (\vect{x_0} + B \vect{x})\in \vect{X} \} $.
Let $\vect{x}\in \vect{X}_B$ then for all  $i$,  $|p_i(B \vect{x}+ \vect{x_0})| \le c_i$ thus $|p_i(B \vect{x})| \le c_i+|p_i(\vect{x_0})|$. 
Then $||Q' B \vect{x}||^2\leq \sum_{i=1}^{l} (c_i + |p_i(\vect{x_0})|)^2 = c$.
From Proposition \ref{WKT}, if $\mu_{min}$ is the smallest eigenvalue of $(Q' B)^* (Q' B)$, we have $\mu_{min}||\vect{x}||^2\leq||Q' B \vect{x}||^2\leq c$.
Since $Q' B$ is invertible, $\mu_{min} \not=0$ and  $||\vect{x}||\leq \sqrt{\frac{c}{\mu_{min}}}$.
Then $\vect{X}_B$ and $\vect{X}$ are finite, and we can easily compute them.
\end{proof}

\begin{proposition}\label{prop:largeabelianfree}
If no ancestor of $[\varepsilon, \ldots,\varepsilon, \nullvect, \ldots, \nullvect]$ by $g$ is realizable by $h$ then $g(\operatorname{Fact^\infty}(h))$ avoids abelian $k$-th power of period larger than $\max_{a\in\alphabet} |g(a)|$.
\end{proposition}
The condition of Proposition \ref{prop:largeabelianfree} can be easily checked by a computer using Proposition \ref{prop:largeperiodset} and Theorem \ref{decidingtemplates}. 
If one wants to decide if $g(\operatorname{Fact^\infty}(h))$ avoids abelian $k$-th powers of period at least $p\le\max_{a\in\alphabet} |g(a)|$, then one can use
Proposition \ref{prop:largeabelianfree} and check if $g(\operatorname{Fact^\infty}(h))$ does not contain an abelian $k$-th power of period $l$ for every $p\le l < \max_{a\in\alphabet} |g(a)|$.
If $p> \max_{a\in\alphabet} |g(a)|$, then one can take a large enough integer $k$ such that $p\le \max_{a\in\alphabet} |g(h^k(a))|$, and make the computation on $g\circ h^k$ instead of $g$.
Note that if $E_e(M_h) \cap \Ker{M_g} = \nullvectset$, then for every $k\in \mathbb{N}$, $E_e(M_h) \cap \Ker{M_{g\circ h^k}} = \nullvectset$.
Otherwise, let $\vect{x} \in (E_e(M_h) \cap \Ker{M_{g\circ h^k}}) \setminus \nullvectset$. Then $M_{g\circ h^k}=M_gM_h^k$, and $M_h^k\vect{x} \in \Ker{M_g}$. Moreover $\vect{x}\in E_e(M_g)$, that is $M_h^k\vect{x} \in E_e(M_g)$ and $M_h^k\vect{x}\ne \nullvect$.
Thus $M_h^k\vect{x} \in E_e(M_h) \cap \Ker{M_g} \setminus \nullvectset$, and we have a contradiction.
Consequently we have the following theorem.

\begin{theorem}\label{th:largeab}
Let $h:\alphabet^*\rightarrow\alphabet^*$ be a primitive morphism with no eigenvalue of absolute value $1$, let $g: \alphabet^* \rightarrow \alphabet'^*$ be a morphism, and let $p,k\in \mathbb{N}$.
If $E_e(M_h) \cap \Ker{M_g}  = \nullvectset$  then one can decide
whether $g(h^\infty(a))$ avoids abelian $k$-th power of period larger than $p$.
\end{theorem}

In Section \ref{sec:results:large}, we present a morphic word over 3 letters which avoids large abelian squares.

\subsection{Deciding if a pure morphic word avoids additive powers on $\mathbb{Z}^d$}

In this part we consider the morphism $\Phi: (\alphabet^*, . ) \rightarrow (\mathbb{Z}^d, +)$ with $d\in \mathbb{N}$.
Let the matrix $F_\Phi$ be such that $\forall w$, $\Phi(w) = F_\Phi \Psi(w)$.

\begin{proposition}\label{prop:insersect}
  If $M_h$ has no eigenvalue of absolute value $1$ and $E_e(M_h) \cap \Ker{\Phi} = \nullvectset$
  then one can compute a finite set of templates $S$  
  such that each $k$-th power modulo $\Phi$ in $\operatorname{Fact^\infty(h)}$ is a realization of a template in $S$.
\end{proposition}
\begin{proof}
One wants to compute the set $S$ of all realizable templates $t=[\varepsilon,\ldots, \varepsilon, \vect{d}_1,\ldots, \vect{d}_{k-1}]$
such that $F_\Phi \vect{d}_i={\nullvect}$.
This set is finite, and we can compute a finite super-set of it exactly as in Proposition \ref{prop:largeperiodset} using Smith normal form of $F_\Phi$ and Proposition \ref{WKT} to bound 
the coefficients of the elements of the basis.
\end{proof}

Proposition \ref{prop:insersect} combined with Theorem \ref{decidingtemplates} give directly the following result which allows us to decide
 $k$-th-power-modulo-$\Phi$ freeness.

\begin{theorem}\label{addfree}
Let $h:\alphabet^*\rightarrow\alphabet^*$ be a primitive morphism with no eigenvalue of absolute value $1$, and let $\Phi: \alphabet^* \rightarrow \mathbb{Z}^d$ a morphism.
If $E_e(M_h) \cap \Ker{\Phi}  = \nullvectset$  then one can decide
whether every word in $\operatorname{Fact^\infty}(h)$ is $k$-th-power-modulo-$\Phi$-free.
\end{theorem}

\section{Results}\label{sec:results}

In this section we use the algorithms described  in Sections  \ref{sec:main} and \ref{sec:app} to show that additive squares are avoidable over $\mathbb{Z}^2$ and that
abelian squares of period more than $5$ are avoidable over the ternary alphabet.
We also give some other new results about additive power avoidability and long $2$-abelian power avoidability.

\subsection{Abelian-square-free pure morphic words}\label{sec:results:main}
Let $h_6$ be the following morphism:  
$$h_6:\left\{
  \begin{array}{llll}
      a \rightarrow & ace \hspace{1cm}& 
      b \rightarrow & adf\\ 
      c \rightarrow & bdf& 
      d \rightarrow & bdc\\ 
      e \rightarrow & afe& 
      f \rightarrow & bce.\\ 
    \end{array}
  \right.
$$

\begin{theorem}\label{abelianthe}
$h_6^\omega(a)$ is abelian-square-free.
\end{theorem}

We show it using the algorithm presented in Section \ref{sec:main}.
We provide a computer program to show Theorem \ref{abelianthe} (see Section \ref{sec:results:large}).

The matrix associated has the following eigenvalues: $0$ (with algebraic multiplicity $3$), $3$, $\sqrt{3}$ and $-\sqrt{3}$.
A Jordan decomposition of $M_{h_6}$ is $P J P^{-1}$, with: 

$$ J=\left[\begin{smallmatrix} 0&1&0&0&0&0\\ \noalign{\medskip}0&0&0&0&0&0\\ \noalign{\medskip}0&0&0&0&0&0\\ \noalign{\medskip}0&0&0&3&0&0
\\ \noalign{\medskip}0&0&0&0&\sqrt {3}&0\\ \noalign{\medskip}0&0&0&0&0&-\sqrt {3}\end{smallmatrix}\right]
\text{ and }
P= \left[\begin{smallmatrix}-\frac{1}{2}&0&-1&1&2+\sqrt {3}&2-\sqrt {3}\\ 
\noalign{\medskip}\frac{1}{2}&-1&0&1&-2-\sqrt {3}&\sqrt {3}-2\\ 
\noalign{\medskip}-\frac{1}{2}&1&-1&1&-1&-1\\ 
\noalign{\medskip}0&0&1&1&-3-2\,\sqrt {3}&2\,\sqrt {3}-3\\ 
\noalign{\medskip}0&\frac{1}{2}&1&1&3+2\,\sqrt {3}&3-2\,\sqrt {3}\\ 
\noalign{\medskip}\frac{1}{2}&-\frac{1}{2}&0&1&1&1 \end{smallmatrix}\right].$$

The bounds on $r^*_i$, $i\in\{1,2,3\}$ computed as explained in the proof of Proposition \ref{bound<1} on $(h_6)^2$, are respectively $4$, $\frac{4}{3}$ and $\frac{4}{3}$.
The template $[\varepsilon,\varepsilon,\varepsilon,\nullvect]$ has 28514 parents with respect to those bounds, and it has 48459 different ancestors including itself.
With this set of template the value of the $s$ from Proposition \ref{factorvsdiscrset} is 44. 
None of the factors of $h_6^\omega(a)$ of size less than 44 realizes a forbidden pattern so we can conclude that $h_6^\omega(a)$ avoids abelian squares.

\bigskip
Let $h_8$ be the following morphism:
$$h_8:
\left\{
  \begin{array}{llll}
      a \rightarrow & h \hspace{1cm}& 
      b \rightarrow & g\\ 
      c \rightarrow & f & 
      d \rightarrow & e\\ 
      e \rightarrow & hc & 
      f \rightarrow & ac\\ 
      g \rightarrow & db &  
      h \rightarrow & eb.\\ 
    \end{array}
  \right.
$$

\begin{theorem}
Words in $h_8^\infty$ (e.g. infinite fixed points of $(h_8)^2$) are abelian-square-free.
\end{theorem}

This morphism may also be interesting because it is a small morphism which gives an abelian-square-free word, 
its matrix is invertible and it has 4 eigenvalues of absolute value less than 1.

\subsection{Additive-square-free words on $\mathbb{Z}^2$}\label{sec:results:addsq}
Let $\Phi $ be the following morphism:
$$\Phi:  
\left\{
  \begin{array}{llll}
      a \rightarrow & ( 1, 0, 0 )\hspace{1cm}& 
      b \rightarrow & ( 1, 1, 1 )\\ 
      c \rightarrow & ( 1, 2, 1 )& 
      d \rightarrow & ( 1, 0, 1 )\\ 
      e \rightarrow & ( 1, 2, 0 )& 
      f \rightarrow & ( 1, 1, 0 ).\\ 
    \end{array}
  \right.
$$
\begin{theorem}
$h_6^\omega(a)$ does not contains squares modulo $\Phi$. 
\end{theorem}
In order to check this Theorem, we provide with this article a code that applies the algorithm described in the previous sections to $\phi(h_6^\omega(a))$.

In other words, the fixed point $h_{\operatorname{add}}^\omega\left({\scriptscriptstyle\begin{pmatrix} 0\\ 0\end{pmatrix}}\right)$ of the following morphism does not contain any additive square.
$$h_{\operatorname{add}}:  
\left\{
\scriptscriptstyle
  \begin{array}{llll}
      \begin{pmatrix} 0\\ 0\end{pmatrix} \rightarrow & \begin{pmatrix} 0\\ 0\end{pmatrix} \begin{pmatrix} 2\\ 1\end{pmatrix} \begin{pmatrix} 2\\ 0\end{pmatrix}\hspace{1cm}& 
      \begin{pmatrix} 1\\ 1\end{pmatrix} \rightarrow & \begin{pmatrix} 0\\ 0\end{pmatrix} \begin{pmatrix} 0\\ 1\end{pmatrix} \begin{pmatrix} 1\\ 0\end{pmatrix}\\ 
      \begin{pmatrix} 2\\ 1\end{pmatrix} \rightarrow & \begin{pmatrix} 1\\ 1\end{pmatrix} \begin{pmatrix} 0\\ 1\end{pmatrix} \begin{pmatrix} 1\\ 0\end{pmatrix}& 
      \begin{pmatrix} 0\\ 1\end{pmatrix} \rightarrow & \begin{pmatrix} 1\\ 1\end{pmatrix} \begin{pmatrix} 0\\ 1\end{pmatrix} \begin{pmatrix} 2\\ 1\end{pmatrix}\\ 
      \begin{pmatrix} 2\\ 0\end{pmatrix} \rightarrow & \begin{pmatrix} 0\\ 0\end{pmatrix} \begin{pmatrix} 1\\ 0\end{pmatrix} \begin{pmatrix} 2\\ 0\end{pmatrix}& 
      \begin{pmatrix} 1\\ 0\end{pmatrix} \rightarrow & \begin{pmatrix} 1\\ 1\end{pmatrix} \begin{pmatrix} 2\\ 1\end{pmatrix} \begin{pmatrix} 2\\ 0\end{pmatrix}.\\ 
    \end{array}
  \right.
$$

It implies the following result:
\begin{theorem}\label{addsquareavoidable}
 $\mathbb{Z}^2$ is not uniformly $2$-repetitive.
\end{theorem}

\subsection{Additive-cubes-free words on $\mathbb{Z}$}\label{sec:results:addcu}
\citeN{avoid3consblock}
show that the fixed point of
$f:  
       0 \rightarrow 03, 
       1 \rightarrow 43, 
       3 \rightarrow 1, 
       4 \rightarrow 01, 
$ 
avoids additive cubes.
Our algorithm concludes that this morphism avoids additive cubes.

\citeN{abeliancube} shows that one can avoid additive cubes on the alphabet $\{0,1,5\}$. 
Thus the only open 4-letters alphabets on small integers are: $\{0,1,2,3\}$, $\{0,1,2,4\}$, $\{0,2,3,5\}$. 
We are able to prove that the answer is positive with the last two alphabets.

Let $h_4:  
\left\{
  \begin{array}{ll}
      0 \rightarrow & 001\\ 
      1 \rightarrow & 041\\ 
      2 \rightarrow & 41\\ 
      4 \rightarrow & 442\\ 
    \end{array}
  \right.
$
and $h'_4:  
\left\{
  \begin{array}{ll}
      0 \rightarrow & 03\\ 
      2 \rightarrow & 53\\ 
      3 \rightarrow & 2\\ 
      5 \rightarrow & 02.\\ 
    \end{array}
  \right.
$
\begin{theorem}
$h_4^\omega(0)$ and $h_4'^\omega(0)$ do not contain any additive cube.
\end{theorem}

It seems easy to find morphisms whose fixed points avoid additive cube for any $4$-letters alphabet, except for $\{0,1,2,3\}$.

\begin{Problem}
 Are additive cubes avoidable over $\{0,1,2,3\}$?
\end{Problem}

\subsection{M\"akel\"a's Problem 1}\label{sec:results:large}
Let $g_3$ be the following morphism:
$$g_3:  
\left\{
  \begin{array}{ll}
      a \rightarrow & \texttt{bbbaabaaac}\\ 
      b \rightarrow & \texttt{bccacccbcc}\\ 
      c \rightarrow & \texttt{ccccbbbcbc}\\ 
      d \rightarrow & \texttt{ccccccccaa}\\ 
      e \rightarrow & \texttt{bbbbbcabaa}\\ 
      f \rightarrow & \texttt{aaaaaaabaa}.\\ 
    \end{array}
  \right.
$$

\begin{theorem}\label{makanswerr}
The word obtained by applying $g_3$ to the fixed point of $h_6$, that is $g_3(h_6^\omega(a))$, does not contain any square of period more than $5$.
\end{theorem}
The kernel of $q_3$ is of dimension 3, but using the bounds on the 3 null eigenvalues of $h_6$ we can compute that $[\varepsilon, \ldots,\varepsilon, \nullvect, \ldots, \nullvect]$ has at most 16214 parents by $g_3$  realizable by $h_6$. 
This is checked using
Theorem \ref{th:largeab}.
This gives an answer to a weak version of Problem \ref{makquestsquares}.
\begin{theorem}\label{makcorr}
There is an infinite word over 3 letters avoiding abelian squares of period more than $5$. 
\end{theorem}

In order to check Theorem \ref{makanswerr}, we provide with this article a computer program that applies the algorithm described in the previous sections to $g_3(h_6^\omega(a))$.
This program also shows Theorem \ref{abelianthe} as a corollary of Theorem \ref{makanswerr}.

\subsection{Avoidability of long $2$-abelian squares}
Recently, Karhum\"aki \emph{et al.} introduced the notion of \emph{$k$-abelian equivalence} as a generalization of both abelian equivalence and equality of words~\cite{Karhumaki1}.
Two words $u$ and $v$ are said \emph{$k$-abelian equivalent} (for $k\ge 1$), denoted $u \approx_{a,k} v$, if for every $w\in \Sigma^*$ such that $|w| \leq k$, $|u|_w=|v|_w$. 
A word $u_1u_2\ldots u_n$ is a \emph{$k$-abelian $n$-th power} if it is non-empty, and $u_1 \approx_{a,k} u_2 \approx_{a,k} \ldots  \approx_{a,k} u_n$. 
Its \emph{period} is $\vert u_1\vert$.
A word is said to be \emph{$k$-abelian-$n$-th-power-free} if none of its factors is a $k$-abelian $n$-th power.
Note that when $k=1$, the $k$-abelian equivalence is exactly the abelian equivalence.

The existence of the  word from Theorem \ref{makanswerr} allows us to answer to the following questions:
\begin{Problem}[\cite{abeliancube,bigkabelian}]\label{question2abeliansquare}
Can we avoid $2$-abelian squares of period at least $p$ on the binary alphabet,
for some $p\in \mathbb{N}$ ? 
\end{Problem}

Let $h_2$ be the following morphism:
$$h_2:  
\left\{
  \begin{array}{ll}
      a \rightarrow & 11100000000\\ 
      b \rightarrow & 11010001010\\ 
      c \rightarrow & 11111101010.\\ 
    \end{array}
  \right.
$$

\begin{theorem}
 $h_2(g_3(h_6^\omega(a)))$ does not contain any $2$-abelian square of period more than 60.
\end{theorem}
Using the same technique as in~\cite{bigkabelian} we can show, by reasoning only on $h_2$, that any $2$-abelian square of $h_2(g_3(h_6^\omega(a)))$ is small (with respect to 60)
or has a parent realized by $g_3(h_6^\omega(a))$ which is an abelian square. 
Thus the largest $2$-abelian squares of $h_2(g_3(h_6^\omega(a)))$ have a period of at most $11\times 5+10= 65$. 
The value 60 is then obtained by checking all the factors of $h_2(g_3(h_6^\omega(a)))$ of size at most $65$.
z
The value $60$ is probably not optimal (the lower bound from \cite{bigkabelian} is $2$).
The easiest way to improve this result would be to improve the upper bound on the period for M\"akel\"a's question.

\printbibliography

\end{document}